\documentclass[11pt,a4paper]{article}

\usepackage[a4paper,margin=2.5cm]{geometry}
\usepackage[T1]{fontenc}
\usepackage[utf8]{inputenc}
\usepackage{lmodern}
\usepackage{microtype}

\usepackage{amssymb,amsmath,amsfonts,latexsym}
\usepackage{amsthm}
\usepackage{mathtools}
\usepackage{bm}

\usepackage{graphicx}
\usepackage{subcaption}
\usepackage{float}
\usepackage{multirow,array}
\usepackage{xcolor}

\usepackage{enumitem}
\allowdisplaybreaks

\usepackage[colorlinks=true,
linkcolor=blue,
citecolor=blue,
urlcolor=blue,
bookmarksnumbered=true]{hyperref}

\newtheorem{theorem}{Theorem}[section]
\newtheorem{lemma}[theorem]{Lemma}
\newtheorem{proposition}[theorem]{Proposition}

\newtheorem{definition}[theorem]{Definition}
\newtheorem{example}[theorem]{Example}
\newtheorem{remark}[theorem]{Remark}

\title{The Stability and Density of Norm Attaining Operators on Reducing Subspaces}

\author{Zinan Su$^{1,*}$ \and Yuanhang Zhang$^{1}$ \\[1ex]
	\small $^1$School of Mathematics, Jilin University, Qianjin Street 2699,
	Changchun 130012, Jilin, China \\
	\small $^*$Corresponding author: \texttt{suzn1023@mails.jlu.edu.cn}}

\date{Received: ...... / Revised: ...... / Accepted: ......}

\begin{document}
	
	\maketitle
	
	\begin{abstract}
		For a bounded linear operator on an infinite-dimensional Hilbert space, we consider the class $\beta(H)$ of operators whose restrictions to all nonzero reducing subspaces attain their norms. We prove that $\beta(H)$ is dense in $\mathcal{B}(H)$ in the operator norm. Since $\beta(H)$ is not stable under arbitrary compact perturbations, we introduce a natural subclass $\beta_0(H)$ and investigate its perturbation properties. We show that $\beta_0(H)$ is stable under finite-rank perturbations and obtain a complete characterization of those compact perturbations that preserve membership in $\beta_0(H)$. We also determine the maximal compactly perturbation-invariant subset of $\beta_0(H)$. These results clarify the spectral structure and stability of norm-attaining operators on reducing subspaces.
	\end{abstract}
	
	
	
	\tableofcontents
	
	\section{Introduction}\label{sec1}
	\setcounter{equation}{0}
	
	Let $H$ be an infinite-dimensional complex Hilbert space and let $\mathcal{B}(H)$ denote the algebra of all bounded linear operators on $H$. An operator $T\in \mathcal{B}(H)$ is said to be norm attaining if there exists a unit vector $x\in H$ such that $\|Tx\|=\|T\|$. This notion traces back to the classical study of norm attaining continuous linear functionals on Banach spaces. In 1961, Bishop and Phelps proved that the set of norm attaining functionals on any Banach space is dense in its dual space \cite{bishop1961}. Bollob\'as subsequently obtained further extensions of the Bishop--Phelps theorem \cite{Bollobas1970}. Lindenstrauss subsequently extended this idea to the operator setting: if the domain space is a reflexive Banach space, then norm attaining operators are dense in the space of all bounded linear operators with respect to the operator norm \cite{Lindenstrauss1963}.
	
	For positive operators on a Hilbert space, norm attainability is equivalent to the norm of the operator being an eigenvalue \cite[Proposition 2.4]{Carvajal2012}. Hence, the norm attaining property provides a powerful tool for locating the largest eigenvalue of an operator, which is of significance in the solution of operator equations and in min-max problems.
	
	Carvajal and Neves introduced in \cite{Carvajal2012} a class of operators with a stronger property than mere norm attainability --- the absolutely norm attaining operators (or $\mathcal{AN}$-operators for short): an operator $T$ is called an $\mathcal{AN}$-operator if, for every non-zero closed subspace $M$ of the domain space, the restriction $T|_M$ is norm attaining. Subsequently, Pandey and Paulsen \cite{Pandey2017} characterized positive $\mathcal{AN}$-operators by means of spectral methods, Ramesh \cite{Ramesh2018} studied paranormal $\mathcal{AN}$-operators, and Venku Naidu and Ramesh \cite{Naidu2019} discussed representations of self-adjoint and normal $\mathcal{AN}$-operators. The class of $\mathcal{AN}$-operators contains all compact operators and all partial isometries with finite-dimensional kernel, but it is not closed under the adjoint operation (see \cite{Ramesh2018}).
	
	It is worth noting that $\mathcal{AN}$-operators are required to be norm attaining when restricted to all closed subspaces --- a rather stringent condition. Ramesh and Osaka \cite{Ramesh2021,wuhu} relaxed this by only requiring norm attainability on reducing subspaces, defining the following class of operators:
	\[
	\beta(H):=\{T\in \mathcal{B}(H):\text{$T|_M$ is norm attaining for every non-zero }M\in\mathcal{R}_T\},
	\]
	where $\mathcal{R}_T$ denotes the collection of all reducing subspaces of $T$. Clearly, $\mathcal{AN}(H)\subseteq\beta(H)\subseteq\mathcal{N}(H)$, where $\mathcal{N}(H)$ denotes the class of norm attaining operators on $H$. It was shown in \cite{Ramesh2021} via explicit examples that both inclusions are strict, and a structure theorem for positive operators in $\beta(H)$ as well as a complete representation of general operators in this class were established.
	
	It is important to note that the equivalence $T\in\beta(H)\iff|T|\in\beta(H)$ stated in \cite[Proposition 3.5(i)]{Ramesh2021} is not correct in this generality. Only the implication
	\[
	|T|\in\beta(H)\implies T\in\beta(H)
	\]
	holds. The converse fails in general; an explicit counterexample is given in Example~\ref{ex:positive-fails} below. This observation motivates the introduction of the subclass
	\[
	\beta_0(H):=\{T\in\mathcal{B}(H):|T|\in\beta(H)\},
	\]
	which satisfies $\beta_0(H)\subseteq\beta(H)$, the inclusion being proper in general. The spectral perturbation theory developed in this paper is naturally formulated for $\beta_0(H)$ rather than for the larger class $\beta(H)$.
	
	A first natural question is how large the class $\beta(H)$ is inside $\mathcal{B}(H)$. Although the Lindenstrauss theorem assures the density of the norm attaining operators $\mathcal{N}(H)$, it does not yield the density of $\beta(H)$, since $\beta(H)$ is a proper subclass of $\mathcal{N}(H)$. We nevertheless prove in Section~8 that $\beta(H)$ is norm-dense in $\mathcal{B}(H)$: by the spectral theorem, every self-adjoint operator is a norm-limit of finite-spectrum self-adjoint operators, and each finite-spectrum self-adjoint operator belongs to $\beta(H)$. This density, however, coexists with a remarkable fragility, as the perturbation results of Sections~4--6 will show.
	
	The present paper builds upon \cite{Ramesh2021} and further investigates the behavior of the class $\beta_0(H)$ under compact perturbations. Since $\mathcal{K}(H)\subseteq\beta_0(H)$, it is natural to ask whether $\beta_0(H)$ is stable under compact perturbations. We show that this is not the case in general---a concrete example is provided in Section~4 (Example~\ref{ex:3})---and then proceed to determine exactly when invariance does hold.
	
	The central question is: given $T\in\beta_0(H)$ and a compact operator $K$, under what conditions does $T+K$ remain in $\beta_0(H)$? Perturbation problems of this kind are of fundamental importance in operator theory: compact operators, being norm-limits of finite-rank operators, represent the most natural ``small perturbation'', and characterizing the invariance (or lack thereof) of $\beta_0(H)$ under these perturbations delineates the stability boundary of this operator class.
	
	In Section~2, we define the basic terminology and notations, correct the modulus proposition, give a counterexample showing that the converse fails, and introduce the subclass $\beta_0(H)$. In Section~3 we examine the structure theorem stated in \cite[Theorem 4.6]{Ramesh2021}, explain the limitation of its $\mathbb{N}$-indexed formulation, and prove a corrected version in terms of reverse well-ordered point spectra for operators in $\beta_0(H)$; this corrected structure theorem supplies the four auxiliary quantities---the limit modulus $c$, the residual operator $D$, the interaction operator $R$, and the total perturbation operator $L$---used throughout the paper. In Section~4 we prove that $\beta_0(H)$ is invariant under finite-rank perturbations (Theorem~\ref{thm:finrank}) using the finite-rank spectral stability lemma; Example~\ref{ex:3} shows that the conclusion fails for general compact operators, even trace-class ones. In Section~5 we develop the general theory for compact perturbations: a full characterization in terms of the pure point spectrum and reverse well-ordered point spectrum of the total perturbation operator $L$ (Theorem~\ref{thm:general}), which simplifies, when the residual operator $D$ is compact, to the condition that $\sigma_p(L)\cap(-\infty,0)$ is finite (Theorem~\ref{thm:main}). Examples~\ref{ex:1} and \ref{ex:2} show that the two hypotheses in the main perturbation result are sharp. In Section~6 we characterize the maximal subset of $\beta_0(H)$ that survives all compact perturbations, proving that $\widehat{\beta_0}(H)=\mathcal{K}(H)$ (Theorem~\ref{thm:hatbeta}). In Section~7 we collect the conclusions that remain valid for the larger class $\beta(H)$ itself. In Section~8 we prove that $\beta(H)$ is dense in $\mathcal{B}(H)$ in the operator norm (Theorem~\ref{thm:density}).
	
	\section{Notations and terminology}\label{sec2}
	\setcounter{equation}{0}
	
	Throughout this paper, $H$ denotes an infinite-dimensional \emph{separable} complex Hilbert space with inner product $\langle\cdot,\cdot\rangle$ and induced norm $\|\cdot\|$. For a subspace $M$ of $H$, $S_M:=\{x\in M:\|x\|=1\}$ is the unit sphere of $M$, and $M^\perp$ denotes its orthogonal complement. If $M$ is closed, $P_M$ stands for the orthogonal projection onto $M$.
	
	Let $\mathcal{B}(H)$ be the algebra of all bounded linear operators on $H$. For $T\in \mathcal{B}(H)$, $N(T)$ and $R(T)$ denote its null space and range, respectively, and $T^*$ its adjoint. An operator $T\in \mathcal{B}(H)$ is called
	\begin{itemize}
		\item \emph{normal} if $T^*T=TT^*$,
		\item \emph{self-adjoint} if $T=T^*$,
		\item \emph{positive} (written $T\geqslant 0$) if $\langle Tx,x\rangle\geqslant 0$ for all $x\in H$.
	\end{itemize}
	Every positive operator admits a unique positive square root $T^{1/2}$.
	For self-adjoint $S,T\in \mathcal{B}(H)$, we write $S\geqslant T$ whenever
	$S-T\geqslant 0$. For a subset $A\subseteq \mathcal{B}(H)$, we denote by
	$A_+$ the set of positive elements of $A$.
	
	An operator $P\in \mathcal{B}(H)$ with $P^2=P$ is a \emph{projection}. If moreover $N(P)\perp R(P)$, then $P$ is an \emph{orthogonal projection}; this is equivalent to $P=P^*$ and also to $P$ being normal. Two orthogonal projections are called \emph{mutually orthogonal} if their ranges are orthogonal.
	
	An operator $V\in \mathcal{B}(H)$ is a \emph{partial isometry} if $V|_{N(V)^\perp}$ is an isometry, and \emph{unitary} if it is an isometry and onto. For any $T\in \mathcal{B}(H)$, the \emph{modulus} of $T$ is $|T|:=(T^*T)^{1/2}$. There exists a unique partial isometry $U\in \mathcal{B}(H)$ such that $T=U|T|$ and $N(U)=N(T)$; this is the \emph{polar decomposition} of $T$.
	
	A closed subspace $M$ of $H$ is \emph{invariant} under $T\in \mathcal{B}(H)$ if $TM\subseteq M$, and \emph{reducing} if both $M$ and $M^\perp$ are invariant under $T$. The collection of all reducing subspaces of $T$ is denoted by $\mathcal{R}_T$. The \emph{spectrum} and \emph{point spectrum} of $T$ are
	\begin{gather*}
		\sigma(T):=\{\lambda\in\mathbb{C}:T-\lambda I\text{ is not invertible in }\mathcal{B}(H)\},\\
		\sigma_p(T):=\{\lambda\in\mathbb{C}:T-\lambda I\text{ is not injective}\}.
	\end{gather*}
	
	For $S\in\mathcal{B}(H)$, its essential spectrum is denoted by $\sigma_{\mathrm{ess}}(S)$ and is defined by
	\[
	\sigma_{\mathrm{ess}}(S):=\{\lambda\in\mathbb{C}:S-\lambda I\text{ is not Fredholm}\}.
	\]
	For self-adjoint operators, the essential spectrum is invariant under compact perturbations; see, for example, Kato~\cite{Kato1995}.
	
	An operator $S\in\mathcal{B}(H)$ is said to have \emph{pure point spectrum} if the closed linear span of its eigenvectors is all of $H$. Equivalently,
	\[
	H=\overline{\operatorname{span}}\bigl\{\ker(S-\lambda I):\lambda\in\sigma_p(S)\bigr\}.
	\]
	For further details on the pure point spectrum and its basic properties, see \cite{864849}.
	
	The ideals of bounded finite-rank and compact operators on $H$ are denoted by $\mathcal{F}(H)$ and $\mathcal{K}(H)$, respectively.
	
	The following order-theoretic notion will play a central role.
	
	\begin{definition}\label{def:rwo}
		A subset $\Lambda$ of $\mathbb{R}$ is said to be \emph{reverse well-ordered} if every non-empty subset of $\Lambda$ has a greatest element.
	\end{definition}
	
	\begin{lemma}[Order-theoretic facts]\label{lem:rwo-facts}
		If $\Lambda\subseteq\mathbb{R}$ is reverse well-ordered, then $\Lambda$ and its closure $\overline{\Lambda}$ are countable. Moreover, if $x\in\overline{\Lambda}\setminus\Lambda$, then there exists a strictly decreasing sequence $\lambda_n\in\Lambda$ such that $\lambda_n\downarrow x$.
	\end{lemma}
	
	\begin{proof}
		Every reverse well-ordered subset of $\mathbb{R}$ is countable. Indeed, under the reverse order $\lambda\preceq\mu$ if and only if $\lambda\geqslant\mu$, the set $\Lambda$ is well-ordered, and every well-ordered subset of the separable metric space $\mathbb{R}$ is countable.
		
		We next prove that $\overline{\Lambda}$ is also reverse well-ordered. Suppose, to the contrary, that $x_1<x_2<\cdots$ is a strictly increasing sequence in $\overline{\Lambda}$. Choose $\lambda_1\in\Lambda$ with $\lambda_1<(x_1+x_2)/2$. For each $n\geqslant2$, choose
		\[
		\lambda_n\in\Lambda\cap\left(\frac{x_n+x_{n-1}}{2},\frac{x_n+x_{n+1}}{2}\right).
		\]
		Such a choice is possible because $x_n\in\overline{\Lambda}$. Then $\lambda_1<\lambda_2<\cdots$, contradicting the reverse well-ordering of $\Lambda$. Hence $\overline{\Lambda}$ is reverse well-ordered and therefore countable.
		
		Now let $x\in\overline{\Lambda}\setminus\Lambda$. We claim that $\Lambda$ contains points arbitrarily close to $x$ from above. Otherwise, for some $\varepsilon>0$, one would have $\Lambda\cap(x,x+\varepsilon)=\varnothing$. Since $x\in\overline{\Lambda}$ and $x\notin\Lambda$, the set $\Lambda\cap(x-\varepsilon,x)$ would then have no greatest element. By the elementary characterization of reverse well-ordering for subsets of $\mathbb{R}$, it would contain a strictly increasing sequence, contradicting the reverse well-ordering of $\Lambda$. Therefore, for every $\varepsilon>0$, there exists $\lambda\in\Lambda$ with $x<\lambda<x+\varepsilon$. Choosing recursively gives a strictly decreasing sequence $\lambda_n\downarrow x$.
	\end{proof}
	
	Equivalently, the reverse order on $\Lambda$ is a well-order. In particular, a reverse well-ordered subset of $\mathbb{R}$ cannot contain a strictly increasing infinite sequence. Conversely, for subsets of $\mathbb{R}$, the absence of an increasing infinite sequence is equivalent to reverse well-ordering. Plainly, every finite set is reverse well-ordered; both $\{1/n:n\in\mathbb{N}^*\}$ and $\{1/n:n\in\mathbb{N}^*\}\cup\{0\}$ are reverse well-ordered.
	
	An operator $T\in \mathcal{B}(H)$ is \emph{norm attaining}, denoted $T\in\mathcal{N}(H)$, if there exists $x\in S_H$ such that $\|Tx\|=\|T\|$. The following characterization for self-adjoint operators is fundamental.
	
	\begin{lemma}[Reducing subspaces and spectral projections]\label{lem:reducing-spectral}
		Let $S=S^*$ and let $M\in\mathcal{R}_S$. Then $P_M$ commutes with every spectral projection $E_S(\Delta)$ of $S$. In particular, $M$ reduces every spectral projection of $S$.
	\end{lemma}
	
	\begin{proof}
		Since $M$ reduces $S$, both $M$ and $M^\perp$ are invariant under $S$. Hence $P_M$ commutes with $S$. By the bounded Borel functional calculus, $P_M$ commutes with $f(S)$ for every bounded Borel function $f$ on $\sigma(S)$. In particular, taking $f=\chi_\Delta$ gives
		\[
		P_M E_S(\Delta)=E_S(\Delta)P_M
		\]
		for every Borel set $\Delta$. Thus $M$ reduces every spectral projection of $S$.
	\end{proof}
	
	\begin{proposition}[{\cite[Proposition 2.4]{Carvajal2012}}]\label{prop:NA}
		Let $T\in \mathcal{B}(H)$ be self-adjoint. Then
		\begin{enumerate}
			\item[(1)] $T\in\mathcal{N}(H)$ if and only if either $\|T\|\in\sigma_p(T)$ or $-\|T\|\in\sigma_p(T)$;
			\item[(2)] if $T\geqslant 0$, then $T\in\mathcal{N}(H)$ if and only if $\|T\|\in\sigma_p(T)$.
		\end{enumerate}
	\end{proposition}
	
	The central object of this paper is the following class, introduced by Ramesh and Osaka \cite{Ramesh2021}:
	\begin{definition}\label{def:beta}
		\[
		\beta(H):=\{T\in \mathcal{B}(H):\text{$T|_M$ is norm attaining for every non-zero }M\in\mathcal{R}_T\}.
		\]
	\end{definition}
	
	The following proposition corrects the equivalence stated in \cite[Proposition 3.5(i)]{Ramesh2021}. Only one implication holds in general.
	
	\begin{proposition}[Corrected modulus proposition]\label{thm:positive}
		Let $T\in\mathcal{B}(H)$. If $|T|\in\beta(H)$, then $T\in\beta(H)$. Equivalently, $T\in\beta(H)$ is necessary for $|T|\in\beta(H)$. The converse implication fails in general.
	\end{proposition}
	
	\begin{proof}
		Let $M\in\mathcal{R}_T$ be non-zero. Since $M$ reduces $T$, it also reduces $|T|$. By hypothesis, $|T|\in\beta(H)$, so there exists $x\in S_M$ such that
		\[
		\||T|x\|=\|\,|T|\big|_M\|.
		\]
		But
		\[
		\|Tx\|=\||T|x\|,\qquad\|T|_M\|=\|\,|T|\big|_M\|,
		\]
		so $T|_M$ attains its norm at $x$. Since $M\in\mathcal{R}_T$ was arbitrary, $T\in\beta(H)$.
	\end{proof}
	
	\begin{example}\label{ex:positive-fails}
		The converse implication in Proposition~\ref{thm:positive} fails. Consider the unilateral weighted shift $T$ on $\ell^2$ defined by
		\[
		Te_n=a_ne_{n+1},\qquad a_1=1,\quad a_n=1-\frac{1}{n}\ (n\geqslant2).
		\]
		The weights are non-zero and pairwise distinct, so $\{T\}'=\mathbb{C}I$ and hence $\mathcal{R}_T=\{0,H\}$. Since $\|T\|=1$ and $\|Te_1\|=\|e_2\|=1$, the operator $T$ attains its norm. Therefore $T\in\beta(H)$.
		
		On the other hand, $|T|e_n=a_ne_n$, so $|T|$ is diagonal with respect to $\{e_n\}$ and
		\[
		\sigma_p(|T|)=\{1\}\cup\left\{1-\frac{1}{n}:n\geqslant2\right\}.
		\]
		The subset $\{1-1/n:n\geqslant2\}$ has no greatest element, so $\sigma_p(|T|)$ is not reverse well-ordered. By Theorem~\ref{thm:structure-corrected}, $|T|\notin\beta(H)$. Thus $T\in\beta(H)$ but $|T|\notin\beta(H)$.
	\end{example}
	
	\begin{definition}\label{def:beta0}
		Define
		\[
		\beta_0(H):=\{T\in\mathcal{B}(H):|T|\in\beta(H)\}.
		\]
		By Proposition~\ref{thm:positive}, $\beta_0(H)\subseteq\beta(H)$, and the inclusion is proper in general (see Example~\ref{ex:positive-fails}).
	\end{definition}
	
	The next structure theorem is the main technical tool throughout Sections~4--6. The original statement in \cite[Theorem 4.6]{Ramesh2021} is discussed in Section~3, where we explain why an $\mathbb{N}$-indexed decreasing representation is not sufficiently general. The corrected formulation in terms of reverse well-ordered point spectra is proved as Theorem~\ref{thm:structure-corrected}, stated for the class $\beta_0(H)$.
	
	Based on this representation, we introduce four auxiliary quantities that will be used systematically throughout the paper.
	
	\begin{definition}\label{def:auxiliary}
		Let $T\in\beta_0(H)$ with $|T|=\sum_{\lambda\in\sigma_p(|T|)}\lambda E_{|T|}(\{\lambda\})$ as in Theorem~\ref{thm:structure-corrected}. Set
		\[
		c:=\inf\sigma_p(|T|),
		\]
		and call it the \emph{limit modulus} of $T$. Since $|T|$ has pure point spectrum, $\sigma(|T|)=\overline{\sigma_p(|T|)}\subseteq[c,\infty)$, so $|T|\geqslant cI$. The number $c$ is finite because $|T|$ is bounded and positive, and it need not belong to $\sigma_p(|T|)$. Define the \emph{residual operator}
		\[
		D:=T^*T-c^2I=\sum_{\lambda\in\sigma_p(|T|)}(\lambda^2-c^2)E_{|T|}(\{\lambda\})\geqslant 0.
		\]
		For a compact perturbation $K\in\mathcal{K}(H)$, define the \emph{interaction operator}
		\[
		R:=T^*K+K^*T+K^*K,
		\]
		which is self-adjoint and compact, and the \emph{total perturbation operator}
		\[
		L:=D+R.
		\]
	\end{definition}
	
	Note that $R$ is self-adjoint and compact, $L$ is self-adjoint, and it is compact precisely when $D$ is compact. Since $D$ has the spectral decomposition
	\[
	D=\sum_{\lambda\in\sigma_p(D)}\lambda Q_\lambda,
	\]
	where $\sigma_p(D)=\{\lambda^2-c^2:\lambda\in\sigma_p(|T|)\}$ and $Q_\lambda$ is the orthogonal projection onto the corresponding eigenspace, the compactness of $D$ is equivalent to the following condition: for every $\varepsilon>0$, the set
	\[
	\{\lambda\in\sigma_p(D):\lambda\geqslant \varepsilon\}
	\]
	is finite and every positive eigenvalue of $D$ has finite multiplicity. These four objects---$c$, $D$, $R$, $L$---form the backbone of all our subsequent analysis. The limit modulus $c$ captures the asymptotic behaviour of the eigenvalues of $|T|$; the residual operator $D$ measures the deviation of $T^*T$ from $c^2 I$ and governs whether the compactness hypothesis in Theorem~\ref{thm:main} is satisfied; the interaction operator $R$ encodes the interplay between $T$ and the perturbation $K$; and the total perturbation operator $L=D+R$ combines both contributions.
	
	For a self-adjoint operator, the spectral theorem provides a functional calculus: if $f$ is a real-valued continuous function on $\sigma(S)$, then $f(S)$ is a well-defined self-adjoint operator. Its spectral measure satisfies $E_{f(S)}(\Delta)=E_S(f^{-1}(\Delta))$ for every Borel set $\Delta$. In particular, every eigenvector of $S$ with eigenvalue $\lambda$ is an eigenvector of $f(S)$ with eigenvalue $f(\lambda)$.
	
	More precisely, for a self-adjoint operator $S\in\mathcal{B}(H)$, the spectral theorem
	asserts the existence of a unique projection-valued measure $E_S$ on the Borel
	$\sigma$-algebra of $\sigma(S)$ such that
	\[
	S=\int_{\sigma(S)}\lambda\,dE_S(\lambda),\qquad
	I=\int_{\sigma(S)}dE_S(\lambda).
	\]
	For each Borel set $\Delta\subseteq\sigma(S)$, $E_S(\Delta)$ is an orthogonal
	projection, and the mapping $\Delta\mapsto E_S(\Delta)$ is countably additive in the
	strong operator topology. When $S$ has pure point spectrum, the spectral measure is
	purely atomic: $E_S(\{\lambda\})$ is the orthogonal projection onto
	$\ker(S-\lambda I)$ for each eigenvalue $\lambda$, and the integral reduces to a
	discrete sum
	\[
	S=\sum_{\lambda\in\sigma_p(S)}\lambda\,E_S(\{\lambda\}),\qquad
	\sum_{\lambda\in\sigma_p(S)}E_S(\{\lambda\})=I.
	\]
	In particular, for a compact self-adjoint operator, the non-zero spectrum consists of
	eigenvalues of finite multiplicity with $0$ as the only possible accumulation point.
	The spectral measure is then supported on the non-zero eigenvalues together with the point $0$, which may belong to the spectrum without being an eigenvalue.
	
	Finally, we recall the standard spectral characterization of compact self-adjoint operators, which will be used repeatedly.
	
	\begin{theorem}[{\cite[Chapter II, Theorem 5.1]{Conway}}]\label{thm:compact}
		Let $T\in \mathcal{B}(H)$ be self-adjoint. Then $T$ is compact if and only if $\sigma(T)\setminus\{0\}=\sigma_p(T)\setminus\{0\}$, every non-zero eigenvalue has finite multiplicity, and the set of non-zero eigenvalues (if infinite) has $0$ as its only accumulation point.
	\end{theorem}
	
	\section{The structure theorem revisited}\label{sec3}
	\setcounter{equation}{0}
	
	The structure theorem for positive operators in $\beta(H)$ is the starting point of the perturbation analysis in this paper. In \cite[Theorem 4.6]{Ramesh2021} it was stated in the following form.
	
	\begin{theorem}[{\cite[Theorem 4.6]{Ramesh2021}}]\label{thm:structure-original}
		Let $T\in\mathcal{B}(H)$. Then $T\in\beta(H)$ if and only if there exist a non-increasing sequence $\{a_n\}_{n=1}^\infty$ in $[0,\infty)$ and a family $\{P_n\}_{n=1}^\infty$ of mutually orthogonal projections such that
		\[
		\sum_{n=1}^\infty P_n=I,
		\]
		and
		\[
		T^*T=\sum_{n=1}^\infty a_n^2 P_n,\qquad |T|=\sum_{n=1}^\infty a_n P_n.
		\]
	\end{theorem}
	
	The $\mathbb{N}$-indexed formulation does not cover the full class considered here. Requiring the eigenvalues to be arranged as a \emph{sequence indexed by the natural numbers} imposes an additional restriction: when $|T|$ has infinitely many positive eigenvalues accumulating at a point which itself is an eigenvalue, the descending arrangement of the eigenvalues may require an ordinal index larger than $\omega$, and hence cannot be fitted into a single $\mathbb{N}$-indexed sequence. The following elementary example exhibits this defect.
	
	\begin{example}\label{ex:counter-structure}
		Let $H=\ell^2\oplus\ell^2$ and let $\{e_n\}_{n=1}^\infty$, $\{f_n\}_{n=1}^\infty$ be orthonormal bases of the first and the second copy of $\ell^2$, respectively. Define a positive diagonal operator $S$ on $H$ by
		\[
		Se_n=\frac{1}{n}e_n,\qquad Sf_n=0,\qquad n\in\mathbb N^*.
		\]
		Then
		\[
		\sigma_p(S)=\left\{\frac1n:n\in\mathbb N^*\right\}\cup\{0\}.
		\]
		Moreover,
		\[
		H=\overline{\operatorname{span}}
		\left(
		\{e_n:n\ge1\}\cup\{f_n:n\ge1\}
		\right),
		\]
		so $S$ has pure point spectrum. We next verify that $\sigma_p(S)$ is reverse well-ordered. Let $\Lambda$ be any nonempty subset of $\sigma_p(S)$. If $\Lambda=\{0\}$, then clearly $\max\Lambda=0$. Otherwise, $\Lambda$ contains a positive element. Define
		\[
		N:=\min\left\{n\in\mathbb N^*:\frac1n\in\Lambda\right\}.
		\]
		Then $1/n\leqslant1/N$ for every $1/n\in\Lambda$, and hence
		\[
		\max\Lambda=\frac1N.
		\]
		Thus $\sigma_p(S)$ is reverse well-ordered. We now verify directly that $S\in\beta(H)_+$. Let $M\in\mathcal R_S\setminus\{0\}$. Since $S$ is self-adjoint and $M$ reduces $S$, the subspace $M$ reduces every spectral projection of $S$. Hence
		\[
		M=\overline{\bigoplus_{\lambda\in\Lambda_M}\bigl(\ker(S-\lambda I)\cap M\bigr)},
		\qquad
		\Lambda_M:=\{\lambda\in\sigma_p(S):\ker(S-\lambda I)\cap M\neq\{0\}\}.
		\]
		Because $M\neq\{0\}$, the set $\Lambda_M$ is nonempty. Since $\sigma_p(S)$ is reverse well-ordered, there exists
		\[
		\lambda_*:=\max\Lambda_M.
		\]
		Choose a unit vector $x_0\in\ker(S-\lambda_*I)\cap M$. Then $Sx_0=\lambda_*x_0$. For every unit vector $x\in M$, write $x=\sum_{\lambda\in\Lambda_M}x_\lambda$ with $x_\lambda\in\ker(S-\lambda I)\cap M$. The eigenspaces corresponding to distinct eigenvalues are mutually orthogonal, and therefore
		\[
		\|Sx\|^2=\sum_{\lambda\in\Lambda_M}\lambda^2\|x_\lambda\|^2\leqslant\lambda_*^2\sum_{\lambda\in\Lambda_M}\|x_\lambda\|^2=\lambda_*^2.
		\]
		Consequently, $\|S|_M\|=\lambda_*=\|Sx_0\|$, so $S|_M$ attains its norm. Since $M$ was arbitrary, $S\in\beta(H)_+$.
		
		We now show that $S$ cannot be represented in the form
		\[
		S=\sum_{n=1}^\infty a_nP_n,
		\]
		where $\{a_n\}_{n=1}^\infty$ is a non-increasing sequence of nonnegative numbers, $\{P_n\}_{n=1}^\infty$ is a family of mutually orthogonal projections, and
		\[
		\sum_{n=1}^\infty P_n=I
		\]
		in the strong operator topology.
		
		Suppose, to the contrary, that such a representation exists. Since the projections are mutually orthogonal, for every $m\in\mathbb N^*$ we have
		\[
		SP_m=\left(\sum_{n=1}^\infty a_nP_n\right)P_m=a_mP_m,
		\qquad
		P_mS=P_m\left(\sum_{n=1}^\infty a_nP_n\right)=a_mP_m.
		\]
		Thus
		\[
		SP_m=P_mS=a_mP_m.
		\]
		In particular,
		\[
		R(P_m)\subseteq\ker(S-a_mI),
		\]
		so every nonzero $P_m$ corresponds to an eigenvalue $a_m$ of $S$.
		
		Fix $k\in\mathbb N^*$. Since $Se_k=\frac1k e_k$ and
		\[
		e_k=\sum_{m=1}^\infty P_me_k,
		\]
		there exists some $m=m(k)$ such that $P_me_k\neq0$. For this $m$,
		\[
		SP_me_k=P_mSe_k=\frac1kP_me_k,
		\]
		while
		\[
		SP_me_k=a_mP_me_k.
		\]
		Hence
		\[
		\left(a_m-\frac1k\right)P_me_k=0.
		\]
		Since $P_me_k\neq0$, we obtain
		\[
		a_m=\frac1k.
		\]
		As $k\in\mathbb N^*$ was arbitrary, the sequence $\{a_n\}$ must contain every number
		\[
		1,\frac12,\frac13,\ldots.
		\]
		In particular, $\{a_n\}$ must contain infinitely many distinct positive terms.
		
		On the other hand, since $Sf_1=0$ and
		\[
		f_1=\sum_{n=1}^\infty P_nf_1\neq0,
		\]
		there exists some $m_0$ such that
		\[
		P_{m_0}f_1\neq0.
		\]
		Then
		\[
		SP_{m_0}f_1=P_{m_0}Sf_1=0,
		\]
		whereas
		\[
		SP_{m_0}f_1=a_{m_0}P_{m_0}f_1.
		\]
		Since $P_{m_0}f_1\neq0$, it follows that
		\[
		a_{m_0}=0.
		\]
		Because $\{a_n\}$ is non-increasing and $a_n\geqslant0$, we therefore have
		\[
		0\leqslant a_n\leqslant a_{m_0}=0,
		\qquad n\geqslant m_0,
		\]
		and hence
		\[
		a_n=0,
		\qquad n\geqslant m_0.
		\]
		Thus only finitely many terms of the sequence $\{a_n\}$ can be positive. This contradicts the fact proved above that $\{a_n\}$ must contain the infinitely many distinct positive numbers
		\[
		1,\frac12,\frac13,\ldots.
		\]
		Therefore no such non-increasing sequence $\{a_n\}$ and mutually orthogonal projections $\{P_n\}$ can exist. Hence
		\[
		S\neq\sum_{n=1}^\infty a_nP_n
		\]
		for every non-increasing sequence $\{a_n\}_{n=1}^\infty$ of nonnegative numbers and every family of mutually orthogonal projections $\{P_n\}_{n=1}^\infty$ satisfying $\sum_{n=1}^\infty P_n=I$.
		
		Since $S\geqslant0$, we have $|S|=S$, and consequently the representation
		\[
		|S|=\sum_{n=1}^\infty a_nP_n
		\]
		is impossible as well. Therefore, although $S\in\beta(H)_+$, it cannot be represented in the $\mathbb N$-indexed form asserted in the original structure theorem. This shows that the $\mathbb N$-indexed formulation is not sufficiently general and cannot describe all operators in $\beta(H)_+$. The correct formulation must instead be expressed in terms of the reverse well-ordering of the point spectrum.
	\end{example}
	
	The difficulty illustrated above is intrinsic, not an artifact of the diagonal example. The source of the problem is that the descending enumeration of a reverse well-ordered set may have an ordinal type exceeding $\omega$ when an eigenvalue occurs after infinitely many larger eigenvalues. This is exactly what happens in Example~\ref{ex:counter-structure}: the positive eigenvalues satisfy $1>1/2>1/3>\cdots\downarrow0$, while $0$ is also an eigenvalue.
	
	The correct formulation of the structure theorem is as follows. We first record the elementary criterion for membership in $\beta(H)_+$.
	
	\begin{proposition}[Reverse well-ordered spectral criterion]\label{prop:reverse}
		Let $S\geqslant0$ have pure point spectrum. If $\sigma_p(S)$ is reverse well-ordered, then $S\in\beta(H)_+$.
	\end{proposition}
	
	\begin{proof}
		Take an arbitrary $M\in\mathcal{R}_S$. By Lemma~\ref{lem:reducing-spectral}, $M$ reduces every spectral projection of $S$. Since $S$ has pure point spectrum, it follows that $M$ is the orthogonal direct sum of the eigenspaces of $S$ cut out by $M$:
		\[
		M=\overline{\bigoplus_{\lambda\in\Lambda_M}\bigl(\ker(S-\lambda I)\cap M\bigr)},
		\qquad
		\Lambda_M:=\{\lambda\in\sigma_p(S):\ker(S-\lambda I)\cap M\neq\{0\}\}.
		\]
		As $M\neq\{0\}$, the set $\Lambda_M$ is non-empty. By the reverse well-orderedness of $\sigma_p(S)$, the maximum
		\[
		\lambda_*:=\max\Lambda_M
		\]
		exists. Pick a unit vector $x\in\ker(S-\lambda_* I)\cap M$. Then $Sx=\lambda_*x$ and hence $\|Sx\|=\lambda_*$. For any unit vector $y\in M$, writing $y=\sum_{\lambda\in\Lambda_M}y_\lambda$ with $y_\lambda\in\ker(S-\lambda I)$, orthogonality gives
		\[
		\|Sy\|^2=\sum_{\lambda\in\Lambda_M}\lambda^2\|y_\lambda\|^2
		\leqslant\lambda_*^2\sum_{\lambda\in\Lambda_M}\|y_\lambda\|^2
		=\lambda_*^2.
		\]
		Thus $\|S|_M\|=\lambda_*$ is attained at $x$. Since $M\in\mathcal{R}_S$ was arbitrary, $S\in\beta(H)_+$.
	\end{proof}
	
	We can now prove the corrected structure theorem for the class $\beta_0(H)$.
	
	\begin{theorem}[Corrected structure theorem for $\beta_0(H)$]\label{thm:structure-corrected}
		Let $T\in\mathcal{B}(H)$. Then $T\in\beta_0(H)$ if and only if $|T|$ has pure point spectrum and $\sigma_p(|T|)$ is reverse well-ordered. In this case,
		\[
		|T|=\sum_{\lambda\in\sigma_p(|T|)}\lambda\,E_{|T|}(\{\lambda\}),
		\]
		where $E_{|T|}(\{\lambda\})$ is the orthogonal projection onto $\ker(|T|-\lambda I)$, and the sum converges in the strong operator topology.
	\end{theorem}
	
	\begin{proof}
		\emph{Necessity.} Suppose $T\in\beta_0(H)$. Then, by definition, $|T|\in\beta(H)$. We prove that $|T|$ has pure point spectrum and that $\sigma_p(|T|)$ is reverse well-ordered.
		
		Write $S:=|T|\geqslant0$ and define
		\[
		H_p:=\overline{\operatorname{span}}\{\ker(S-\lambda I):\lambda\in\sigma_p(S)\},
		\qquad H_c:=H_p^\perp.
		\]
		Since $S$ is self-adjoint, $H_p$ and $H_c$ are reducing subspaces for $S$, and $S|_{H_c}$ has no eigenvalues. If $H_c\neq\{0\}$, then $\|S|_{H_c}\|>0$: otherwise $S|_{H_c}=0$, so every nonzero vector in $H_c$ would be an eigenvector corresponding to $0$, contradicting the definition of $H_c$. Therefore, since $S|_{H_c}\geqslant0$, Proposition~\ref{prop:NA}(2) implies that $S|_{H_c}$ does not attain its norm. This contradicts $S\in\beta(H)$. Hence $H_c=\{0\}$, and $S$ has pure point spectrum.
		
		Now assume, for contradiction, that $\sigma_p(S)$ is not reverse well-ordered. Then there exists a non-empty subset of $\sigma_p(S)$ without a greatest element. Choosing recursively from this subset gives pairwise distinct eigenvalues
		\[
		\lambda_1<\lambda_2<\cdots
		\]
		and, since $S$ is bounded and self-adjoint, $\lambda_n\nearrow s$ for some $s\in\mathbb{R}$. Choose unit eigenvectors $x_n$ with $Sx_n=\lambda_nx_n$. Since distinct eigenspaces are orthogonal, the subspace $M:=\overline{\operatorname{span}}\{x_n:n\geqslant1\}$ is reducing for $S$, and
		\[
		\sigma_p(S|_M)=\{\lambda_n:n\geqslant1\}.
		\]
		In particular, $s\notin\sigma_p(S|_M)$, because the only eigenvalues of $S|_M$ are the selected $\lambda_n$. On the other hand,
		\[
		\|S|_M\|=\sup_{n\geqslant1}\lambda_n=s.
		\]
		Hence $S|_M$ does not attain its norm, contradicting $S\in\beta(H)$. Therefore $\sigma_p(S)$ is reverse well-ordered.
		
		\emph{Sufficiency.} Suppose $|T|$ has pure point spectrum and $\sigma_p(|T|)$ is reverse well-ordered. By Proposition~\ref{prop:reverse}, $|T|\in\beta(H)_+$, hence $|T|\in\beta(H)$, and therefore $T\in\beta_0(H)$.
	\end{proof}
	
	Several remarks are in order.
	
	\begin{remark}\label{rem:compatibility}
		The original $\mathbb{N}$-indexed formulation in \cite[Theorem 4.6]{Ramesh2021} is recovered as a special case. If $\sigma_p(|T|)$ is a reverse well-ordered set of order type at most $\omega$, then its elements can be enumerated as a non-increasing sequence $a_1\geqslant a_2\geqslant\cdots$, and one recovers the representation $|T|=\sum_{n=1}^\infty a_nP_n$. This happens, for instance, when the decreasing enumeration of the eigenvalues has no accumulation point inside $\sigma_p(|T|)$, or when the only such accumulation point does not belong to $\sigma_p(|T|)$. The reverse well-ordered formulation is exactly the general statement.
	\end{remark}
	
	\begin{remark}\label{rem:limit-modulus}
		For $T\in\beta_0(H)$, let $c:=\inf\sigma_p(|T|)$. This quantity, the limit modulus, is independent of any particular enumeration of the eigenvalues. It need not be an eigenvalue, as Example~\ref{ex:1} illustrates. It is the natural replacement of $\lim_{n\to\infty}a_n$ in the original sequence formulation.
	\end{remark}
	
	\begin{remark}\label{rem:residual}
		The residual operator $D:=T^*T-c^2I\geqslant0$ also has pure point spectrum, since $|T|$ does, and
		\[
		\sigma_p(D)=\{\lambda^2-c^2:\lambda\in\sigma_p(|T|)\}.
		\]
		Moreover, since $\sigma_p(|T|)$ is reverse well-ordered and the map $\lambda\mapsto\lambda^2-c^2$ is strictly increasing on $[c,\infty)$, the point spectrum $\sigma_p(D)$ is also reverse well-ordered. This fact will be used crucially in the proof of the finite-rank perturbation theorem in Section~4.
	\end{remark}
	
	For the remainder of the paper, we always use the corrected structure theorem (Theorem~\ref{thm:structure-corrected}) and the auxiliary quantities introduced in Definition~\ref{def:auxiliary}, i.e. $c=\inf\sigma_p(|T|)$, $D=T^*T-c^2I$, $R=T^*K+K^*T+K^*K$, and $L=D+R$.
	
	\section{Finite-rank perturbations}\label{sec4}
	\setcounter{equation}{0}
	
	We begin with the strongest positive result: $\beta_0(H)$ is invariant under finite-rank perturbations. We first establish the spectral fact needed for the proof.
	
	\begin{lemma}[Rank-one gap lemma]\label{lem:rank-one-gap}
		Let $S=S^*\in\mathcal{B}(H)$ and let $J=(a,b)\subseteq\rho(S)$ be an open interval in the resolvent set. If $R=R^*$ has rank one, then $S+R$ has at most one spectral point in $J$.
	\end{lemma}
	
	\begin{proof}
		There are $\gamma\in\mathbb{R}$ and a unit vector $u\in H$ such that
		\[
		R=\gamma(u\otimes u),\qquad (u\otimes u)x:=\langle x,u\rangle u.
		\]
		If $\gamma=0$, the conclusion is immediate. For $\lambda\in J$, since $S-\lambda I$ is invertible, the equation
		\[
		(S+R-\lambda I)x=0
		\]
		forces
		\[
		x=-\gamma\langle x,u\rangle(S-\lambda I)^{-1}u.
		\]
		A nonzero solution must satisfy $\langle x,u\rangle\neq0$, and taking the inner product with $u$ gives
		\[
		1+\gamma\,m(\lambda)=0,
		\qquad
		m(\lambda):=\langle(S-\lambda I)^{-1}u,u\rangle.
		\]
		Conversely, any solution of this scalar equation produces a nonzero eigenvector $(S-\lambda I)^{-1}u$.
		Since $J\subseteq\rho(S)$ and $S=S^*$, the resolvent $(S-\lambda I)^{-1}$ is self-adjoint for real $\lambda\in J$, so $m$ is real-valued and differentiable on $J$, and
		\[
		m'(\lambda)=\langle(S-\lambda I)^{-2}u,u\rangle
		=\|(S-\lambda I)^{-1}u\|^2>0.
		\]
		Thus $m$ is strictly increasing on $J$, so the scalar equation has at most one solution. Moreover, $R$ is compact, and Weyl's theorem gives
		\[
		\sigma_{\mathrm{ess}}(S+R)=\sigma_{\mathrm{ess}}(S).
		\]
		Since $J\subseteq\rho(S)$, we have $J\cap\sigma_{\mathrm{ess}}(S+R)=\varnothing$. Hence every spectral point of the self-adjoint operator $S+R$ lying in $J$ is an isolated eigenvalue of finite multiplicity. Therefore the scalar equation above accounts for every spectral point in $J$, and $S+R$ has at most one spectral point in $J$.
	\end{proof}
	
	\begin{lemma}[Finite-rank spectral stability]\label{lem:finite-rank-stability}
		Let $S=S^*\in\mathcal{B}(H)$ have pure point spectrum and suppose that $\sigma_p(S)$ is reverse well-ordered. If $R=R^*$ has finite rank, then $S+R$ has pure point spectrum and $\sigma_p(S+R)$ is reverse well-ordered.
	\end{lemma}
	
	\begin{proof}
		Put $A:=S+R$. We first prove that $A$ has pure point spectrum. Since $S$ has pure point spectrum,
		\[
		\sigma(S)=\overline{\sigma_p(S)}.
		\]
		By Lemma~\ref{lem:rwo-facts}, $\overline{\sigma_p(S)}$ is countable. Since $R$ is compact, Weyl's theorem on invariance of the essential spectrum under compact perturbations gives
		\[
		\sigma_{\mathrm{ess}}(A)=\sigma_{\mathrm{ess}}(S)\subseteq\sigma(S),
		\]
		so $\sigma_{\mathrm{ess}}(A)$ is countable. Every point of $\sigma(A)\setminus\sigma_{\mathrm{ess}}(A)$ is an isolated eigenvalue of finite multiplicity. The corresponding eigenspaces are mutually orthogonal, and $H$ is separable, so there are at most countably many such eigenvalues. Therefore $\sigma(A)$ is countable. Since a projection-valued measure on a countable set is the strong sum of its singleton projections, the spectral measure of $A$ is purely atomic. Hence $A$ has pure point spectrum.
		
		It remains to prove reverse well-ordering. Suppose, to the contrary, that $\sigma_p(A)$ is not reverse well-ordered. Then there are pairwise distinct eigenvalues
		\[
		\mu_1<\mu_2<\cdots\nearrow\mu.
		\]
		Thus $\mu$ is a limit point of $\sigma(A)$. For a self-adjoint operator, every spectral point outside the essential spectrum is an isolated eigenvalue of finite multiplicity. Hence
		\[
		\mu\in\sigma_{\mathrm{ess}}(A)=\sigma_{\mathrm{ess}}(S).
		\]
		
		Consider
		\[
		\Lambda:=\sigma_p(S)\cap(-\infty,\mu).
		\]
		If $\Lambda=\varnothing$, then $\sigma_p(S)\cap(-\infty,\mu)=\varnothing$ and we choose any $a<\inf\sigma(S)$ with $a<\mu$. If $\Lambda\neq\varnothing$, reverse well-ordering gives a greatest element
		\[
		a:=\max\Lambda.
		\]
		In either case,
		\[
		(a,\mu)\cap\sigma_p(S)=\varnothing.
		\]
		Since $S$ has pure point spectrum, $\sigma(S)=\overline{\sigma_p(S)}$. If some $\lambda\in(a,\mu)$ belonged to $\sigma(S)$, then every neighborhood of $\lambda$ would meet $\sigma_p(S)$. Choosing a sufficiently small neighborhood contained in $(a,\mu)$ would contradict
		\[
		(a,\mu)\cap\sigma_p(S)=\varnothing.
		\]
		Hence
		\[
		(a,\mu)\subseteq\rho(S).
		\]
		
		If $R=0$, there is nothing to prove in the following argument. Otherwise, by the finite-dimensional spectral theorem, there are orthonormal vectors $u_1,\dots,u_r$ and nonzero real numbers $\gamma_1,\dots,\gamma_r$ such that
		\[
		R=\sum_{j=1}^r\gamma_j(u_j\otimes u_j).
		\]
		Set $R_j:=\gamma_j(u_j\otimes u_j)$ and $S_0:=S$, $S_j:=S_{j-1}+R_j$. We claim inductively that $S_j$ has only finitely many spectral points in $(a,\mu)$. This is true for $j=0$. Suppose it is true for $j-1$, and let $\nu_1,\dots,\nu_m$ be all spectral points of $S_{j-1}$ in $(a,\mu)$. Since $S_{j-1}-S$ has finite rank, Weyl's theorem gives $\sigma_{\mathrm{ess}}(S_{j-1})=\sigma_{\mathrm{ess}}(S)$. Because $(a,\mu)\subseteq\rho(S)$, we have $(a,\mu)\cap\sigma_{\mathrm{ess}}(S_{j-1})=\varnothing$. Thus each $\nu_i$ is an isolated eigenvalue of finite multiplicity. Choose pairwise disjoint open intervals $(\ell_i,r_i)\subset(a,\mu)$ around the $\nu_i$ such that $\sigma(S_{j-1})\cap(\ell_i,r_i)=\{\nu_i\}$. The complement
		\[
		(a,\mu)\setminus\bigcup_{i=1}^m[\ell_i,r_i]
		\]
		is a finite union of open intervals contained in $\rho(S_{j-1})$. By Lemma~\ref{lem:rank-one-gap}, each such interval contains at most one spectral point of $S_j$. Moreover, each endpoint $\ell_i$ and $r_i$ belongs to $\rho(S_{j-1})$. Since the resolvent set is open, each endpoint has a small open neighborhood contained in $\rho(S_{j-1})$, and Lemma~\ref{lem:rank-one-gap} shows that each such neighborhood contains at most one spectral point of $S_j$. Thus the endpoints contribute only finitely many additional spectral points.
		
		For each $i$, the eigenspace of $S_j$ at $\nu_i$ is finite-dimensional. Indeed, if $x\in\ker(S_j-\nu_iI)$, then
		\[
		(S_{j-1}-\nu_iI)x=-R_jx\in\operatorname{ran}R_j.
		\]
		The kernel of the map $x\mapsto R_jx$ on $\ker(S_j-\nu_iI)$ is contained in $\ker(S_{j-1}-\nu_iI)$, which is finite-dimensional, while $\operatorname{ran}R_j$ is one-dimensional. Hence
		\[
		\dim\ker(S_j-\nu_iI)\leqslant
		\dim\ker(S_{j-1}-\nu_iI)+1<\infty.
		\]
		The two sides $(\ell_i,\nu_i)$ and $(\nu_i,r_i)$ are resolvent intervals for $S_{j-1}$, so Lemma~\ref{lem:rank-one-gap} gives at most one spectral point of $S_j$ on each side. Together with the possible spectral point at $\nu_i$ itself and the finitely many complement intervals, this yields only finitely many spectral points of $S_j$ in $(a,\mu)$. This completes the induction.
		
		Taking $j=r$ contradicts the infinitely many eigenvalues $\mu_n\in(a,\mu)$ for all sufficiently large $n$. Hence $\sigma_p(A)$ is reverse well-ordered.
	\end{proof}
	
	\begin{theorem}\label{thm:finrank}
		Let $T\in\beta_0(H)$ and let $K$ be a finite-rank operator. Then $T+K\in\beta_0(H)$. In particular, $T+K\in\beta(H)$.
	\end{theorem}
	
	\begin{proof}
		Set $A:=T+K$. By Definition~\ref{def:auxiliary},
		\[
		A^*A=T^*T+T^*K+K^*T+K^*K=c^2I+D+R=c^2I+L,
		\]
		where $R$ is finite-rank and self-adjoint because $K$ is finite-rank. Since $T\in\beta_0(H)$, we have $|T|\in\beta(H)$, and Theorem~\ref{thm:structure-corrected} and Remark~\ref{rem:residual} imply that $D$ has pure point spectrum and $\sigma_p(D)$ is reverse well-ordered. Lemma~\ref{lem:finite-rank-stability}, applied to $S=D$ and $R$, shows that $L=D+R$ has pure point spectrum and $\sigma_p(L)$ is reverse well-ordered.
		
		Since $A^*A=c^2I+L\geqslant0$, we have $L\geqslant-c^2I$ and hence
		\[
		\sigma(L)\subseteq[-c^2,\infty).
		\]
		Therefore the function $f(\lambda)=\sqrt{c^2+\lambda}$ is continuous on $\sigma(L)$. By the spectral theorem,
		\[
		|A|=(c^2I+L)^{1/2}=f(L).
		\]
		Every eigenvector of $L$ is therefore an eigenvector of $|A|$, and
		\[
		\sigma_p(|A|)=\{\sqrt{c^2+\lambda}:\lambda\in\sigma_p(L)\}.
		\]
		Because $f$ is strictly increasing on $[-c^2,\infty)$, it preserves reverse well-orderedness. Hence $|A|$ has pure point spectrum and reverse well-ordered point spectrum. By Theorem~\ref{thm:structure-corrected}, $|A|\in\beta(H)$, i.e. $A=T+K\in\beta_0(H)$. By Proposition~\ref{thm:positive}, $T+K\in\beta(H)$.
	\end{proof}
	
	Theorem~\ref{thm:finrank} shows that the class $\beta_0(H)$ is robust under finite-rank perturbations. One might naturally ask whether the same holds for arbitrary compact perturbations. The following example shows that the answer is negative---even a trace-class perturbation can push an operator out of $\beta_0(H)$.
	
	\begin{example}\label{ex:3}
		Let $H=\ell^2\oplus\mathbb{C}$ be the orthogonal direct sum of $\ell^2$ and $\mathbb{C}$, and let $P_1$, $P_2$ be the orthogonal projections onto $\ell^2$ and $\mathbb{C}$ respectively; note that $P_1P_2=P_2P_1=0$ and $P_1+P_2=I$. Define $T=\sqrt{2}P_1+P_2$; then $T^*T=2P_1+P_2$ and $\sigma_p(T)=\{\sqrt{2},1\}$. Since $\sigma_p(T)$ is finite, it is reverse well-ordered, and by Theorem~\ref{thm:structure-corrected}, $T\in\beta_0(H)$.
		
		On an orthonormal basis $\{e_n\}_{n=1}^\infty$ of $P_1H=\ell^2$, define
		\[
		Ke_n=k_n e_n,\qquad k_n=-\sqrt{2}+\sqrt{2-\frac{1}{2^n}},\qquad n\in\mathbb{N}^*,
		\]
		and set $K=0$ on $P_2H=\mathbb{C}$. Since $K$ is diagonal and real with respect to the orthogonal decomposition, it is self-adjoint. Rationalizing $k_n$,
		\[
		k_n=\frac{-1}{\sqrt{2}+\sqrt{2-\frac{1}{2^n}}}\cdot\frac{1}{2^n},
		\]
		so $|k_n|\leqslant\frac{1}{\sqrt{2}\cdot2^n}$. Hence $K$ is trace-class and therefore compact.
		
		Set $A=T+K$. Since $T$ and $K$ are self-adjoint, $A$ is self-adjoint. On $P_1H$,
		\[
		Ae_n=(\sqrt{2}+k_n)e_n=\sqrt{2-\frac{1}{2^n}}\,e_n,
		\]
		and on $P_2H$, $Af=f$, where $f$ is a unit vector spanning $\mathbb{C}$. Since $P_1H\perp P_2H$, the two eigenspaces are orthogonal. Thus $A=|A|$ is a positive self-adjoint operator with
		\[
		\sigma_p(|A|)=\Bigl\{\lambda_n:=\sqrt{2-\frac{1}{2^n}}: n\in\mathbb{N}^*\Bigr\}\cup\{1\}.
		\]
		
		Consider $M:=\overline{\operatorname{span}}\{e_n:n\geqslant2\}$. Since $|A|$ is diagonal with respect to $\{e_n\}\cup\{f\}$, both $M$ and $M^\perp$ are invariant under $|A|$, so $M\in\mathcal{R}_{|A|}$. On $M$,
		\[
		\sqrt{\frac74}<\sqrt{\frac{15}{8}}<\sqrt{\frac{31}{16}}<\cdots\nearrow\sqrt{2},
		\]
		and hence
		\[
		\|\,|A||_M\,\|=\sup_{n\geqslant2}\lambda_n=\sqrt{2}.
		\]
		But $\sqrt{2}\notin\sigma_p(|A|)$, so the norm is not attained on $M$. Therefore $|A|\notin\beta(H)$, i.e. $A\notin\beta_0(H)$. By Proposition~\ref{thm:positive}, $A\notin\beta(H)$.
	\end{example}
	
	\section{Compact perturbations of $\beta_0(H)$}\label{sec5}
	\setcounter{equation}{0}
	
	In this section we keep the notations introduced in Sections~2 and~3: $T\in\beta_0(H)$ has limit modulus $c=\inf\sigma_p(|T|)$, residual operator $D=T^*T-c^2I$, and, for a compact operator $K\in\mathcal{K}(H)$, $R$ and $L$ are the interaction and total perturbation operators defined in Definition~\ref{def:auxiliary}. Example~\ref{ex:3} has already shown that $T+K\notin\beta_0(H)$ can occur even for trace-class $K$. We now determine exactly when invariance holds.
	
	\begin{theorem}\label{thm:general}
		Let $T\in\beta_0(H)$ and $K\in\mathcal{K}(H)$. Then
		\[
		T+K\in\beta_0(H)
		\iff
		\text{$L$ has pure point spectrum and }\sigma_p(L)\text{ is reverse well-ordered}.
		\]
	\end{theorem}
	
	\begin{proof}
		Set $A:=T+K$. Then
		\[
		A^*A=T^*T+R=c^2I+D+R=c^2I+L,\qquad |A|=(c^2I+L)^{1/2}.
		\]
		
		\emph{Necessity.} Suppose $A\in\beta_0(H)$. Then $|A|\in\beta(H)$. By Theorem~\ref{thm:structure-corrected}, $|A|$ has pure point spectrum and $\sigma_p(|A|)$ is reverse well-ordered. Let
		\[
		|A|=\sum_{\mu\in\sigma_p(|A|)}\mu E_{|A|}(\{\mu\})
		\]
		be its spectral decomposition. Squaring gives
		\[
		c^2I+L=\sum_{\mu\in\sigma_p(|A|)}\mu^2E_{|A|}(\{\mu\}),
		\]
		hence
		\[
		L=\sum_{\mu\in\sigma_p(|A|)}(\mu^2-c^2)E_{|A|}(\{\mu\}).
		\]
		It follows that $L$ has pure point spectrum and
		\[
		\sigma_p(L)=\{\mu^2-c^2:\mu\in\sigma_p(|A|)\}.
		\]
		Since the map $\mu\mapsto\mu^2-c^2$ is strictly increasing on $[0,\infty)$ and $\sigma_p(|A|)$ is reverse well-ordered, $\sigma_p(L)$ is reverse well-ordered.
		
		\emph{Sufficiency.} Suppose $L$ has pure point spectrum and $\sigma_p(L)$ is reverse well-ordered. Write
		\[
		L=\sum_{\lambda\in\sigma_p(L)}\lambda Q_\lambda,
		\]
		where $Q_\lambda$ is the orthogonal projection onto $\ker(L-\lambda I)$. From $A^*A=c^2I+L\geqslant0$, we have $c^2+\lambda\geqslant0$ for every $\lambda\in\sigma_p(L)$. Define
		\[
		\mu_\lambda:=\sqrt{c^2+\lambda},\qquad \lambda\in\sigma_p(L).
		\]
		By the functional calculus,
		\[
		|A|=(c^2I+L)^{1/2}=\sum_{\lambda\in\sigma_p(L)}\mu_\lambda Q_\lambda.
		\]
		Thus $|A|$ has pure point spectrum and $\sigma_p(|A|)=\{\mu_\lambda:\lambda\in\sigma_p(L)\}$. Conversely, if $|A|x=\mu x$ for some $x\ne0$, then
		\[
		Lx=(\mu^2-c^2)x.
		\]
		Thus the displayed equality of point spectra is exact. Since the map $\lambda\mapsto\sqrt{c^2+\lambda}$ is strictly increasing on $[-c^2,\infty)$, the set $\sigma_p(|A|)$ is reverse well-ordered. By Theorem~\ref{thm:structure-corrected}, $|A|\in\beta(H)$, i.e. $A=T+K\in\beta_0(H)$.
	\end{proof}
	
	\begin{remark}\label{rem:spectral}
		Theorem~\ref{thm:general} admits a natural interpretation via the spectral theorem
		for self-adjoint operators. Since $L$ is self-adjoint, the spectral theorem
		supplies a unique projection-valued measure $E_L$ on $\sigma(L)$ such that
		\[
		L=\int_{\sigma(L)}\lambda\,dE_L(\lambda),\qquad
		I=\int_{\sigma(L)}dE_L(\lambda).
		\]
		The condition in the theorem is precisely the requirement that the spectral
		measure $E_L$ be purely atomic and that the atoms be arranged in a reverse
		well-ordered fashion according to their eigenvalues. This condition is not
		automatic: even when $L$ is compact and has purely atomic spectrum, the point
		spectrum need not be reverse well-ordered, as Example~\ref{ex:1} below shows.
	\end{remark}
	
	\begin{remark}\label{rem:general}
		Theorem~\ref{thm:general} provides a complete characterization in full generality, but the condition ``$L$ has pure point spectrum and $\sigma_p(L)$ is reverse well-ordered'' involves the entire spectral decomposition of the total perturbation operator and may be difficult to verify directly. The next theorem shows that when the residual operator $D$ is compact, this condition admits a drastic simplification.
	\end{remark}
	
	\begin{theorem}\label{thm:main}
		Let $T\in\beta_0(H)$ and $K\in\mathcal{K}(H)$. Suppose $D$ is compact. Then
		\[
		T+K\in\beta_0(H) \iff \text{$\sigma_p(L)\cap(-\infty,0)$ is finite.}
		\]
	\end{theorem}
	
	\begin{proof}
		Set $A:=T+K$. Since $D$ and $R$ are compact self-adjoint operators, $L=D+R$ is compact and self-adjoint. By Theorem~\ref{thm:general},
		\[
		A\in\beta_0(H)
		\iff
		\text{$L$ has pure point spectrum and $\sigma_p(L)$ is reverse well-ordered.}
		\]
		The compact spectral theorem gives pure point spectrum automatically. It remains to characterize reverse well-ordering. For a compact self-adjoint operator, every nonzero spectral point is an eigenvalue of finite multiplicity and $0$ is the only possible accumulation point. In particular, the positive eigenvalues form a reverse well-ordered set. If $\sigma_p(L)\cap(-\infty,0)$ is finite, adding these finitely many negative eigenvalues and the possible eigenvalue $0$ preserves reverse well-ordering. Conversely, if $\sigma_p(L)\cap(-\infty,0)$ is infinite, compactness forces these distinct negative spectral values to converge to $0$ from below, and hence they contain a strictly increasing sequence converging to $0$. Therefore $\sigma_p(L)$ is not reverse well-ordered. The result follows.
	\end{proof}
	
	The following example shows that if the total perturbation operator has infinitely many negative eigenvalues, then $T+K\notin\beta_0(H)$, demonstrating the necessity of the finiteness condition in Theorem~\ref{thm:main}.
	
	\begin{example}\label{ex:1}
		Let $H=\ell^2$ with standard orthonormal basis $\{e_n\}_{n=1}^\infty$. Define
		\[
		Te_n=\sqrt{1+\frac{1}{n^4}}\,e_n,\qquad n\in\mathbb{N}^*.
		\]
		Then $T^*Te_n=\bigl(1+\frac{1}{n^4}\bigr)e_n$. Let $a_n=\sqrt{1+\frac{1}{n^4}}$ and let $P_n$ be the orthogonal projection onto $\operatorname{span}\{e_n\}$. Then $\{a_n\}$ is non-negative and decreasing and $\sum_{n=1}^\infty P_n=I$. Since $\sigma_p(T)=\{a_n:n\in\mathbb{N}^*\}$ is reverse well-ordered (it is a decreasing sequence with infimum $1$, and $1\notin\sigma_p(T)$), Theorem~\ref{thm:structure-corrected} gives $T\in\beta_0(H)$. Direct computation yields the limit modulus $c=\inf\sigma_p(T)=1$ and the residual operator $D=T^*T-I=\sum_{n=1}^\infty \frac{1}{n^4}P_n$, which is compact.
		
		Now define
		\[
		Ke_n:=-\frac{1}{n^2}e_n,\qquad n\in\mathbb{N}^*.
		\]
		Then $\{-\frac{1}{n^2}\}$ are simple eigenvalues tending to $0$, so $K$ is a compact self-adjoint operator. Since $T$ and $K$ are both diagonal with respect to $\{e_n\}$, they commute. Hence the interaction operator is $R=T^*K+K^*T+K^*K=2TK+K^2$, and
		\[
		Re_n=\Bigl[2\sqrt{1+\frac{1}{n^4}}\Bigl(-\frac{1}{n^2}\Bigr)+\frac{1}{n^4}\Bigr]e_n.
		\]
		For every $n\in\mathbb{N}^*$, the eigenvalue of the total perturbation operator $L=D+R$ in the direction $e_n$ is
		\[
		\lambda_n=\frac{2}{n^4}-\frac{2}{n^2}\sqrt{1+\frac{1}{n^4}}.
		\]
		With $t_n:=1/n^2$, this can be written as
		\[
		\lambda_n=-2\frac{t_n}{\sqrt{1+t_n^2}+t_n}.
		\]
		The function
		\[
		h(t):=\frac{t}{\sqrt{1+t^2}+t},\qquad t>0,
		\]
		is strictly increasing, since
		\[
		h'(t)=\frac{1}{\sqrt{1+t^2}\,\bigl(\sqrt{1+t^2}+t\bigr)^2}>0.
		\]
		As $t_n\downarrow0$, it follows that
		\[
		\lambda_1<\lambda_2<\cdots<0,\qquad \lambda_n\to0.
		\]
		Thus $L$ has infinitely many distinct negative eigenvalues.
		
		Set $A=T+K$. Then $A$ is self-adjoint and
		\[
		Ae_n=\Bigl(\sqrt{1+\frac{1}{n^4}}-\frac{1}{n^2}\Bigr)e_n.
		\]
		Define $\alpha_n:=\sqrt{1+\frac{1}{n^4}}-\frac{1}{n^2}>0$. Then $A$ is a positive self-adjoint operator, i.e., $|A|=A$. Put $t_n:=1/n^2$ and $f(t):=\sqrt{1+t^2}-t$. Since
		\[
		f'(t)=\frac{t}{\sqrt{1+t^2}}-1<0,\qquad t>0,
		\]
		and $t_n$ decreases strictly to $0$, the sequence $\{\alpha_n\}=\{f(t_n)\}$ increases strictly to $1$. Moreover,
		\[
		\alpha_n^2=1+\frac{1}{n^4}-\frac{2}{n^2}\sqrt{1+\frac{1}{n^4}}+\frac{1}{n^4}
		=1+\frac{2}{n^4}-\frac{2}{n^2}\sqrt{1+\frac{1}{n^4}}<1,\qquad n\geqslant 2,
		\]
		so $\alpha_n<1$ for $n\geqslant 2$.
		
		Set $M:=\overline{\operatorname{span}}\{e_n:n\geqslant2\}$. Since $|A|$ is diagonal with respect to $\{e_n\}$, both $M$ and $M^\perp$ are invariant under $|A|$, so $M\in\mathcal{R}_{|A|}$. We have $\sigma_p(|A||_M)=\{\alpha_n:n\geqslant2\}$, strictly increasing to $1$, hence
		\[
		\|\,|A||_M\,\|=\sup_{n\geqslant2}\alpha_n=1.
		\]
		If the norm were attainable, there would exist a unit vector $x=\sum_{n=2}^\infty\xi_n e_n\in M$ with
		\[
		\sum_{n=2}^\infty\alpha_n^2|\xi_n|^2=1.
		\]
		But $\alpha_n<1$ for all $n$, so
		\[
		\sum_{n=2}^\infty\alpha_n^2|\xi_n|^2<\sum_{n=2}^\infty|\xi_n|^2=1,
		\]
		a contradiction. Therefore $|A|\notin\beta(H)$, i.e. $A\notin\beta_0(H)$.
		
		Note that in this example, for every $p\geqslant1$ we have
		\[
		\sum_{n=1}^\infty |{-1/n^2}|^p=\sum_{n=1}^\infty\frac{1}{n^{2p}}<\infty.
		\]
		Thus $K$ belongs to every Schatten $p$-class $\mathcal{S}_p(H)$ for $p\geqslant1$. Thus even for Schatten $p$-class perturbations, the finiteness condition on negative eigenvalues of the total perturbation operator remains necessary.
	\end{example}
	
	We now return to Example~\ref{ex:3} and identify exactly why Theorem~\ref{thm:main} did not apply there---the residual operator $D$ was not compact.
	
	\begin{example}\label{ex:2}
		Revisit the construction of Example~\ref{ex:3}: $T=\sqrt{2}P_1+P_2$, the limit modulus is $c=1$, and the residual operator is $D=T^*T-I=P_1$, which has infinite rank and is therefore not compact. By definition, $T|_{P_1H}=\sqrt{2}I$ commutes with $K$, so the interaction operator is $R=T^*K+K^*T+K^*K=2TK+K^2$. On $P_1H$,
		\[
		Re_n=(2\sqrt{2}k_n+k_n^2)e_n
		=\Bigl(-\sqrt{2}+\sqrt{2-\frac{1}{2^n}}\Bigr)\Bigl(\sqrt{2}+\sqrt{2-\frac{1}{2^n}}\Bigr)e_n
		=-\frac{e_n}{2^n}.
		\]
		Thus the total perturbation operator $L=D+R=P_1+R$ satisfies, on $P_1H$,
		\[
		Le_n=\Bigl(1-\frac{1}{2^n}\Bigr)e_n,
		\]
		with eigenvalues $\frac12,\frac34,\frac78,\frac{15}{16},\dots$, all positive. On $P_2H$, $L=0$. Hence $L$ has no negative eigenvalues whatsoever.
		
		In summary, for this example we have: $T\in\beta_0(H)$, $K$ is trace-class, $D$ is not compact, $L$ has zero negative eigenvalues, yet $T+K\notin\beta_0(H)$ (as proved in Example~\ref{ex:3}). This demonstrates that the compactness hypothesis on the residual operator $D$ in Theorem~\ref{thm:main} cannot be removed, even if the perturbing operator belongs to every Schatten $p$-class and the total perturbation operator has no negative eigenvalues at all.
	\end{example}
	
	\section{The invariant core}\label{sec6}
	\setcounter{equation}{0}
	
	The results of the preceding sections raise a natural question: which operators in $\beta_0(H)$ survive every compact perturbation? Define
	\[
	\widehat{\beta_0}(H):=\{T\in\beta_0(H): T+K\in\beta_0(H)\text{ for every }K\in\mathcal{K}(H)\}.
	\]
	The next theorem identifies the invariant core exactly.
	
	\begin{theorem}\label{thm:hatbeta}
		$\widehat{\beta_0}(H)=\mathcal{K}(H)$.
	\end{theorem}
	
	\begin{proof}
		Let $T\in\mathcal{K}(H)$. For every $K\in\mathcal{K}(H)$, the operator $T+K$ is compact. Its modulus $|T+K|$ is compact and positive, so its non-zero eigenvalues have finite multiplicity and can only accumulate at $0$. Hence its point spectrum is reverse well-ordered, and Theorem~\ref{thm:structure-corrected} yields $T+K\in\beta_0(H)$. Therefore
		\[
		\mathcal{K}(H)\subseteq\widehat{\beta_0}(H).
		\]
		
		Now suppose $T\in\widehat{\beta_0}(H)\setminus\mathcal{K}(H)$. Since $T=U|T|$ is the polar decomposition and $U$ is bounded, compactness of $|T|$ would imply compactness of $T$. Hence $|T|$ is not compact. By Theorem~\ref{thm:structure-corrected}, $|T|$ has pure point spectrum and reverse well-ordered point spectrum. If $\sigma_{\mathrm{ess}}(|T|)\subseteq\{0\}$, then every nonzero spectral point is an isolated eigenvalue of finite multiplicity and $0$ is the only possible accumulation point. By Theorem~\ref{thm:compact}, this would imply that $|T|$ is compact, a contradiction. Hence there exists
		\[
		a>0\qquad\text{with}\qquad a\in\sigma_{\mathrm{ess}}(|T|).
		\]
		Put $E_a:=\ker(|T|-aI)$. We distinguish two exhaustive cases according to $\dim E_a$.
		
		\noindent\textbf{Case 1: $\dim E_a=\infty$.}
		
		Choose an orthonormal sequence $\{e_j\}_{j=1}^\infty$ in $E_a$. Define a compact self-adjoint operator $K$ by
		\[
		K e_j=-\frac{a}{2(j+1)}e_j,\qquad
		K|_{\operatorname{span}\{e_j:j\geqslant1\}^{\perp}}=0.
		\]
		Put $A:=|T|+K$ and $M:=\overline{\operatorname{span}}\{e_j:j\geqslant1\}$. Then $A\geqslant0$, $M\in\mathcal{R}_A$, and
		\[
		Ae_j=a\left(1-\frac{1}{2(j+1)}\right)e_j.
		\]
		Thus the eigenvalues of $A|_M$ increase strictly to $a$ and remain below $a$. Consequently,
		\[
		\|A|_M\|=a,
		\]
		but $a$ is not an eigenvalue of $A|_M$, so $A|_M$ does not attain its norm. Hence $A\notin\beta(H)_+$, i.e. $A\notin\beta_0(H)$.
		
		\noindent\textbf{Case 2: $\dim E_a<\infty$.}
		
		Since $a\in\sigma_{\mathrm{ess}}(|T|)$ but $E_a$ is finite-dimensional (possibly $E_a=\{0\}$), $a$ cannot be an isolated spectral point of $|T|$. As $|T|$ has pure point spectrum,
		\[
		\sigma(|T|)=\overline{\sigma_p(|T|)},
		\]
		so $a$ is an accumulation point of $\sigma_p(|T|)\setminus\{a\}$. Reverse well-ordering excludes accumulation from below: otherwise there would exist a strictly increasing sequence of eigenvalues converging to $a$, contradicting the fact that every nonempty subset of $\sigma_p(|T|)$ has a greatest element. Hence the accumulation must occur from above, and there are distinct eigenvalues
		\[
		a_1>a_2>\cdots>a,
		\qquad a_j\downarrow a.
		\]
		Choose unit vectors $e_j\in\ker(|T|-a_jI)$ and positive numbers $\varepsilon_j<a/2$ with $\varepsilon_j\downarrow0$. Define $K$ by
		\[
		K e_j=(a-a_j-\varepsilon_j)e_j,
		\qquad
		K|_{\operatorname{span}\{e_j:j\geqslant1\}^{\perp}}=0.
		\]
		Since $a_j\to a$ and $\varepsilon_j\to0$, the coefficients of $K$ tend to zero, so $K$ is compact and self-adjoint. Put $A:=|T|+K$ and $M:=\overline{\operatorname{span}}\{e_j:j\geqslant1\}$. On $E_{a_j}\ominus\operatorname{span}\{e_j\}$ the operator $A$ acts as $a_jI$, while on $e_j$ it acts as $a-\varepsilon_j$. On every other eigenspace of $|T|$, it acts as $|T|$. Thus $A\geqslant0$ and $M\in\mathcal{R}_A$. Moreover,
		\[
		Ae_j=(a-\varepsilon_j)e_j,
		\]
		and the eigenvalues $a-\varepsilon_j$ increase strictly to $a$. Therefore
		\[
		\|A|_M\|=a,
		\]
		while $a\notin\sigma_p(A|_M)$. Hence $A\notin\beta(H)_+$, i.e. $A\notin\beta_0(H)$.
		
		In either case, $K$ vanishes on $\ker T$, while all modified eigenvalues remain strictly positive. Thus $A$ vanishes on $\ker T$ and is injective on $(\ker T)^\perp$, so
		\[
		\ker A=\ker|T|=\ker T.
		\]
		Since $A=A^*$,
		\[
		\operatorname{ran}A\subseteq(\ker T)^\perp,
		\qquad
		\operatorname{ran}K\subseteq(\ker T)^\perp.
		\]
		Let $T=U|T|$ be the polar decomposition and set $C:=UK$. Then $C$ is compact and
		\[
		T+C=U(|T|+K)=UA.
		\]
		Since $A\geqslant0$ and $\operatorname{ran}A\subseteq(\ker T)^\perp$, we have
		\[
		(T+C)^*(T+C)=A U^*U A=A^2,
		\qquad |T+C|=A\notin\beta(H).
		\]
		Hence $T+C\notin\beta_0(H)$, contradicting $T\in\widehat{\beta_0}(H)$. Therefore $\widehat{\beta_0}(H)\subseteq\mathcal{K}(H)$, and equality follows.
	\end{proof}
	
	\section{Conclusions for $\beta(H)$}\label{sec:beta}
	\setcounter{equation}{0}
	
	The perturbation results of the preceding sections are formulated for the subclass $\beta_0(H)$, which is the natural setting for the spectral analysis. In this section we collect what can be said about the larger class $\beta(H)$ itself. We begin with an exact characterization of $\beta(H)$ that involves only the modulus and the reducing subspaces of $T$, and then explain how it differs from the condition $|T|\in\beta(H)$.
	
	
	\begin{proposition}[Characterization of $\beta(H)$]\label{prop:beta-characterization}
		Let $T\in\mathcal{B}(H)$. Then
		\[
		T\in\beta(H)
		\iff
		\text{$|T|$ attains its norm on every reducing subspace of $T$.}
		\]
		Equivalently,
		\[
		T\in\beta(H)
		\iff
		\forall\,M\in\mathcal{R}_T\setminus\{0\},\quad
		\|\,|T|\big|_M\|\in\sigma_p\!\big(|T|\big|_M\big).
		\]
	\end{proposition}
	
	\begin{proof}
		For any $M\in\mathcal{R}_T$, the subspace $M$ also reduces $|T|$, and for every $x\in M$,
		\[
		\|Tx\|=\bigl\||T|x\bigr\|.
		\]
		Hence
		\[
		\|T|_M\|=\bigl\||T|\big|_M\bigr\|,
		\]
		and $T|_M$ attains its norm if and only if $|T|\big|_M$ attains its norm. Therefore
		\[
		T\in\beta(H)
		\iff
		\forall\,M\in\mathcal{R}_T\setminus\{0\},\quad
		|T|\big|_M\text{ attains its norm}.
		\]
		Since $|T|\big|_M$ is a positive operator on $M$, it attains its norm if and only if its norm is an eigenvalue (Proposition~\ref{prop:NA}(2)), i.e.
		\[
		\|\,|T|\big|_M\|\in\sigma_p\!\big(|T|\big|_M\big).
		\]
		This proves the claim.
	\end{proof}
	
	\begin{remark}\label{rem:beta-vs-beta0}
		Proposition~\ref{prop:beta-characterization} must be distinguished from the condition $|T|\in\beta(H)$ studied in the previous sections. The latter requires $|T|$ to attain its norm on every reducing subspace of $|T|$, whereas the former only requires this on the (generally smaller) family of reducing subspaces of $T$. Since
		\[
		\mathcal{R}_T\subseteq\mathcal{R}_{|T|},
		\]
		the condition $|T|\in\beta(H)$ is strictly stronger than $T\in\beta(H)$. Equivalently, the implication $|T|\in\beta(H)\Rightarrow T\in\beta(H)$ of Proposition~\ref{thm:positive} holds, while the converse fails, as Example~\ref{ex:positive-fails} shows. The subclass $\beta_0(H)=\{T:|T|\in\beta(H)\}$ therefore consists precisely of those operators $T\in\beta(H)$ for which the two reducing-subspace families yield the same norm-attainability, in the sense that $|T|$ attains its norm on the larger family $\mathcal{R}_{|T|}$ as well.
	\end{remark}
	
	\begin{remark}\label{rem:structural-consequences}
		Two structural consequences of Proposition~\ref{prop:beta-characterization} deserve emphasis.
		
		\emph{(i) Invariance under isometries.} The condition depends only on the reducing subspaces of $T$ and on the values $\|Tx\|=\||T|x\|$. In particular, if $W\in\mathcal{B}(H)$ is unitary, then $\mathcal{R}_{W^*TW}=W^*\mathcal{R}_T W$, and it follows immediately that
		\[
		T\in\beta(H)\iff W^*TW\in\beta(H).
		\]
		The class $\beta(H)$ is thus invariant under unitary conjugation.
		
		\emph{(ii) Reduction to the modulus.} The operator $T$ enters the characterization only through $|T|$ and through the family $\mathcal{R}_T$. This is the structural content of Proposition~\ref{thm:positive} in its correct form: the modulus completely governs the values of $\|Tx\|$ on each reducing subspace, but the relevant collection of subspaces is $\mathcal{R}_T$, not $\mathcal{R}_{|T|}$. Whenever $\mathcal{R}_T=\mathcal{R}_{|T|}$, the condition collapses to $|T|\in\beta(H)$ and coincides with the description given in Theorem~\ref{thm:structure-corrected}.
	\end{remark}
	
	\section{Norm density of $\beta(H)$}\label{sec7}
	\setcounter{equation}{0}
	
	In this section we show that, although $\beta(H)$ is a proper subclass of the norm attaining operators $\mathcal{N}(H)$, it is nevertheless dense in $\mathcal{B}(H)$ with respect to the operator norm. This stands in sharp contrast with the results of Sections~4--6, which show that the subclass $\beta_0(H)$ is highly fragile under compact perturbations. The idea of the proof is taken from the work of Nag and Ramesh \cite{nag2026}, where the analogous density theorem was established for the class $\mathcal{M}_r(H)$ of minimum attaining operators on reducing subspaces by combining the spectral theorem with the polar decomposition. We begin with two lemmas.
	
	\begin{lemma}\label{lem:finite}
		Let $T\in\mathcal{B}(H)$ be self-adjoint with finite spectrum. Then $T\in\beta(H)$.
	\end{lemma}
	
	\begin{proof}
		By the spectral theorem, $T=\sum_{j=1}^{m}\lambda_j P_j$, where $\lambda_1,\dots,\lambda_m$ are distinct real numbers and $P_1,\dots,P_m$ are mutually orthogonal projections with $\sum_{j=1}^{m}P_j=I$.
		
		Let $M\in\mathcal{R}_T$ be arbitrary. By Lemma~\ref{lem:reducing-spectral}, $M$ reduces each spectral projection $P_j$ of $T$, so $P_jM\subseteq M$. Set
		\[
		J:=\{j:P_j M\neq\{0\}\}.
		\]
		As $M\neq\{0\}$, we have $J\neq\varnothing$. Choose $j_0\in J$ with $|\lambda_{j_0}|=\max_{j\in J}|\lambda_j|$, and take a unit vector $x_0\in P_{j_0}M$. Then $Tx_0=\lambda_{j_0}x_0$, whence $\|Tx_0\|=|\lambda_{j_0}|$.
		
		For any unit vector $y\in M$, we have $y=\sum_{j\in J}P_j y$, and by orthogonality,
		\[
		\|Ty\|^2=\sum_{j\in J}|\lambda_j|^2\|P_j y\|^2
		\leqslant|\lambda_{j_0}|^2\sum_{j\in J}\|P_j y\|^2
		=|\lambda_{j_0}|^2\|y\|^2=|\lambda_{j_0}|^2.
		\]
		Hence $\|T|_M\|=|\lambda_{j_0}|=\|Tx_0\|$, so $T|_M$ is norm attaining. Since $M\in\mathcal{R}_T$ was arbitrary, $T\in\beta(H)$.
	\end{proof}
	
	\begin{lemma}\label{lem:approx}
		Let $T\in\mathcal{B}(H)$ be self-adjoint. Then for every $\varepsilon>0$ there exists a self-adjoint operator $T_\varepsilon$ with finite spectrum such that $\|T-T_\varepsilon\|<\varepsilon$.
	\end{lemma}
	
	\begin{proof}
		By the spectral theorem, $T=\int_{\sigma(T)}\lambda\,dE(\lambda)$. Fix $\varepsilon>0$ and choose a finite Borel partition $\Delta_1,\dots,\Delta_n$ of the compact set $\sigma(T)$ such that $\operatorname{diam}(\Delta_k)<\varepsilon$ for every $k$. For each $k$ choose $\lambda_k\in\Delta_k$ and define
		\[
		T_\varepsilon:=\sum_{k=1}^{n}\lambda_k E(\Delta_k).
		\]
		The projections $E(\Delta_k)$ are mutually orthogonal and sum to $I$, so $T_\varepsilon$ is self-adjoint with finite spectrum. Moreover,
		\[
		T-T_\varepsilon=\int_{\sigma(T)}\Bigl[\lambda-\sum_{k=1}^{n}\lambda_k\chi_{\Delta_k}(\lambda)\Bigr]\,dE(\lambda),
		\]
		and for $\lambda\in\Delta_k$ we have $|\lambda-\lambda_k|\leqslant\operatorname{diam}(\Delta_k)<\varepsilon$. Therefore
		\[
		\|T-T_\varepsilon\|\leqslant\sup_{\lambda\in\sigma(T)}\Bigl|\lambda-\sum_{k=1}^{n}\lambda_k\chi_{\Delta_k}(\lambda)\Bigr|<\varepsilon.
		\]
	\end{proof}
	
	\begin{theorem}\label{thm:density}
		$\beta(H)$ is dense in $\mathcal{B}(H)$ with respect to the operator norm.
	\end{theorem}
	
	\begin{proof}
		If $T=0$, the assertion is immediate. Assume $T\neq0$ and let $\varepsilon>0$. Consider the polar decomposition $T=U|T|$, where $U$ is a partial isometry with $U^*U=P_{(\ker T)^\perp}$. Set $H_0:=(\ker T)^\perp$ and let $A:H_0\to H_0$ be the restriction of $|T|$ to $H_0$; then $A$ is a positive operator and $|T|=0\oplus A$ with respect to the decomposition $H=\ker T\oplus H_0$.
		
		By Lemma~\ref{lem:approx} applied to the self-adjoint operator $A$, there exists a self-adjoint operator $A_\varepsilon$ on $H_0$ with finite spectrum such that $\|A-A_\varepsilon\|<\varepsilon$.
		
		With respect to the decomposition $H=\ker T\oplus H_0$, define $B\in\mathcal{B}(H)$ by $B|_{\ker T}=0$ and $B|_{H_0}=A_\varepsilon$, and set $S:=UB$. Since $U^*U=P_{H_0}$ and $B$ maps $H$ into $H_0$ vanishing on $\ker T$, we have $U^*UB=B$, whence
		\[
		S^*S=B^*U^*UB=B^*B=B^2,
		\]
		because $B$ is self-adjoint. Hence $|S|=(S^*S)^{1/2}=|B|=0\oplus|A_\varepsilon|$. As $A_\varepsilon$ has finite spectrum, $\sigma(|A_\varepsilon|)=\{|\lambda|:\lambda\in\sigma(A_\varepsilon)\}$ is finite, so $|S|$ is a positive operator with finite spectrum. By Lemma~\ref{lem:finite}, $|S|\in\beta(H)_+$, hence $|S|\in\beta(H)$, and by Proposition~\ref{thm:positive}, $S\in\beta(H)$.
		
		Finally,
		\[
		\|T-S\|=\|U|T|-UB\|\leqslant\|U\|\,\||T|-B\|\leqslant\||T|-B\|
		=\|A-A_\varepsilon\|<\varepsilon,
		\]
		where we used $\|U\|\leqslant 1$ and the fact that $|T|-B$ vanishes on $\ker T$ and coincides with $A-A_\varepsilon$ on $H_0$. Thus every $T\in\mathcal{B}(H)$ is a norm-limit of elements of $\beta(H)$, i.e., $\beta(H)$ is dense in $\mathcal{B}(H)$.
	\end{proof}
	
	The results of this paper give a structural description of the class $\beta(H)$ and its behaviour under compact perturbations. On the one hand, Proposition~\ref{prop:beta-characterization} provides an exact characterization of $\beta(H)$ in terms of the modulus and the reducing subspaces of $T$, and Theorem~\ref{thm:density} shows that $\beta(H)$ is norm-dense in $\mathcal{B}(H)$, even though it is a proper subclass of the norm attaining operators $\mathcal{N}(H)$. On the other hand, the behavior of the subclass $\beta_0(H)$ under compact perturbations reveals a remarkable fragility. Theorem~\ref{thm:finrank} shows that finite-rank perturbations are always harmless, offering a stark contrast with the general compact case illustrated by Example~\ref{ex:3}. Theorem~\ref{thm:general} gives a full characterization in terms of the pure point spectrum of the total perturbation operator $L$ being reverse well-ordered, and when the residual operator $D$ is compact, Theorem~\ref{thm:main} reduces this to the simple condition that $\sigma_p(L)\cap(-\infty,0)$ is finite. Examples~\ref{ex:1} and~\ref{ex:2} demonstrate that both conditions (the finiteness of negative eigenvalues of $L$ and the compactness of $D$) are sharp and cannot be relaxed, even when the perturbing operator belongs to every Schatten $p$-class. Finally, Theorem~\ref{thm:hatbeta} reveals that the largest subset of $\beta_0(H)$ stable under \emph{all} compact perturbations is precisely $\mathcal{K}(H)$ itself: for any non-compact operator in $\beta_0(H)$, there always exists a compact perturbation that pushes it out of the class. The gap between $\beta_0(H)$ and the larger class $\beta(H)$ is precisely accounted for by the difference between the two reducing-subspace families $\mathcal{R}_T$ and $\mathcal{R}_{|T|}$, as explained in Remark~\ref{rem:beta-vs-beta0}.
	
	\section*{Acknowledgements}
	
	The authors would like to thank Professor Youqing Ji for the valuable comments and suggestions on this paper.


\begin{thebibliography}{99}
		\normalsize
		
		\bibitem{bishop1961}
		Bishop, E., Phelps, R.R.: A proof that every Banach space is subreflexive.
		Bull. Amer. Math. Soc. \textbf{67}, 97--98 (1961)
		
		\bibitem{Bollobas1970}
		Bollob\'as, B.: An extension to the theorem of Bishop and Phelps.
		Bull. London Math. Soc. \textbf{2}, 181--182 (1970)
		
		\bibitem{Carvajal2012}
		Carvajal, X., Neves, W.: Operators that achieve the norm.
		Integral Equations Operator Theory \textbf{72}, 179--195 (2012)
		
		\bibitem{Conway}
		Conway, J.B.: A Course in Functional Analysis.
		2nd ed., Graduate Texts in Mathematics, Vol. 96, Springer, New York (1990)
		
		\bibitem{Kato1995}
		Kato, T.: Perturbation Theory for Linear Operators.
		2nd ed., Springer-Verlag, Berlin (1995)
		
		\bibitem{Lindenstrauss1963}
		Lindenstrauss, J.: On operators which attain their norm.
		Israel J. Math. \textbf{1}, 139--148 (1963)
		
		\bibitem{nag2026}
		Nag, P., Ramesh, G.: Minimum attaining operators on reducing subspaces: Spectral structure and density.
		arXiv:2608.19286 [math.FA] (2026)
		
		\bibitem{Naidu2019}
		Venku Naidu, D., Ramesh, G.: On absolutely norm attaining operators.
		Proc. Indian Acad. Sci. Math. Sci. \textbf{129}(4), Art. 54, 17 pp. (2019)
		
		\bibitem{Pandey2017}
		Pandey, S.K., Paulsen, V.I.: A spectral characterization of AN operators.
		J. Aust. Math. Soc. \textbf{102}, 369--391 (2017)
		
		\bibitem{Ramesh2018}
		Ramesh, G.: Absolutely norm attaining paranormal operators.
		J. Math. Anal. Appl. \textbf{465}, 547--556 (2018)
		
		\bibitem{Ramesh2021}
		Ramesh, G., Osaka, H.: On a subclass of norm attaining operators.
		Acta Sci. Math. (Szeged) \textbf{87}, 247--263 (2021)
		
		\bibitem{wuhu}
		Ramesh, G., Osaka, H.: On operators which attain their norm on every reducing subspace.
		Ann. Funct. Anal. \textbf{13}, 19 (2022). DOI: 10.1007/s43034-022-00167-8
		
		\bibitem{864849}
		Disintegrating By Parts: Complete ONS and pure point spectrum.
		Mathematics Stack Exchange, version 2014-07-12.
		\url{https://math.stackexchange.com/q/864849}
		
		\bibitem{ReedSimon}
		Reed, M., Simon, B.: Methods of Modern Mathematical Physics I:
		Functional Analysis.
		Revised and enlarged ed., Academic Press, New York (1980).
		
	\end{thebibliography}
\end{document}